\documentclass[12pt,a4paper]{article}
\usepackage[dvips]{graphicx}
\usepackage{epsf}
\usepackage{amssymb}
\usepackage[english]{babel}
\usepackage[T1]{fontenc}
\usepackage{pstricks}
\usepackage{pst-node}
\usepackage{lineno}
\usepackage[normalem]{ulem}
\usepackage{bezier}
\usepackage{epic}
\usepackage{eepic}
\usepackage{color}
\usepackage{graphicx}
\usepackage{xcolor}
\usepackage{enumerate}
\usepackage{enumitem}
\usepackage{wrapfig}

\usepackage{epic}\usepackage{eepic}

\usepackage{mathrsfs}
\usepackage{amsmath}
\usepackage{amsfonts}

\usepackage{tabularx}
\usepackage{booktabs}
\usepackage{multirow}

\newenvironment{block}[3]{%
  \vspace{2mm}\par%
  \refstepcounter{#2}%
  \noindent\textbf{%
    #1 \arabic{#2}.
  }%
}{%
  \par%
}

\newenvironment{thm}[1][]{%
  \begin{block}{\textsc{Theorem}}{theorem}{#1}%
  \slshape%
}{%
  \end{block}%
  \vspace{2mm}%
}
\newenvironment{lem}[1][]{%
  \begin{block}{\textsc{Lemma}}{theorem}{#1}%
  \slshape%
}{%
  \end{block}%
  \vspace{2mm}%
}

\newenvironment{cor}[1][]{%
  \begin{block}{\textsc{Corollary}}{theorem}{#1}%
  \slshape%
}{%
  \end{block}%
  \vspace{2mm}%
}

\newenvironment{proof}[1][Proof]{%
  \par\noindent\emph{#1}.
}{%
  \eop
  \bigskip%
}

\renewcommand\caption[1]{\small\refstepcounter{figure}%
\begin{center}\textbf{Figure \thefigure .}\ #1\end{center}\normalsize}

\newcommand{\eop}{\hspace*{\fill}\nolinebreak$\Box$\nolinebreak\par}
\newcommand{\old}[1]{}

\newrgbcolor{gray5}{0.5 0.5 0.5}
\newrgbcolor{gray7}{0.7 0.7 0.7}
\newrgbcolor{gray9}{0.9 0.9 0.9}

\newrgbcolor{gray97}{0.97 0.97 0.97}
\newrgbcolor{gray92}{0.92 0.92 0.92}
\newrgbcolor{gray87}{0.87 0.87 0.87}

\def\cS{{\cal S}}

\def\cC{{\cal C}}

\def\cB{{\cal B}}

\definecolor{PZgreen}{RGB}{0,158,0}

\date{\today}
\begin{document}

\begin{center}
        \noindent\textbf{\Large
        {Graphs with equal domination\\[3mm] and covering numbers}}
\end{center}

\vspace{10mm}
\begin{center}
\textsc{Andrzej Lingas${}^1$, Mateusz Miotk${}^2$, Jerzy Topp${}^{2,3}$, and Pawe\l{}~\.Zyli\'nski${}^2$ }
\\[3mm]
${}^1$ Lund University, 221-00 Lund, Sweden\\
{\small \texttt{ Andrzej.Lingas@cs.lth.se}}
\\[2mm]
${}^2$ University of Gda\'{n}sk, 80-308 Gda\'nsk, Poland\\[2mm]
${}^3$ The State University of Applied Sciences in Elbl{\k{a}}g, 82-300 Elbl{\k{a}}g, Poland \\

{\small \texttt{\{mmiotk,j.topp,zylinski\}@inf.ug.edu.pl}}
\end{center}

\medskip
\begin{abstract}
\noindent A dominating set of a graph $G$ is a set $D\subseteq V_G$ such that every vertex in $V_G-D$ is adjacent to at least one vertex in $D$, and the domination number $\gamma(G)$ of $G$ is the minimum cardinality of a dominating set of $G$. A set $C\subseteq V_G$ is a covering set of $G$ if every edge of $G$ has at least one vertex in $C$. The covering number $\beta(G)$ of $G$ is the minimum cardinality of a covering set of $G$. The set of connected graphs~$G$ for which $\gamma(G)=\beta(G)$ is denoted by $\cC_{\gamma=\beta}$, while $\cB$ denotes the set of all connected bipartite graphs in which the domination number is equal to the cardinality of the smaller partite set. In this paper, we provide alternative characterizations of graphs belonging to  $\cC_{\gamma=\beta}$ and $\cB$. Next, we present a quadratic time algorithm for recognizing bipartite graphs belonging to $\cB$, and, as a side result, we conclude that the algorithm of Arumugam et al.~\cite{AJBT13} allows to recognize all the graphs belonging to the set $\cC_{\gamma=\beta}$ in quadratic time either. Finally, we consider the related problem of patrolling grids with mobile guards, and show that this problem can be solved in $O(n \log n + m)$ time, where $n$ is the number of line segments of the input grid and $m$ is the number of its intersection points.

{\bf \AmS \; Subject Classification:} 05C69, 05C70
\end{abstract}

\section{Introduction and notation}\label{sec:intro}

In this paper, we follow the notation of~\cite{ChLZ15}. Let $G=(V_G,E_G)$ be a graph with vertex set $V_G$ and edge set $E_G$. For a vertex $v$ of $G$, its {\em neighborhood\/}, denoted by $N_{G}(v)$, is the set of all vertices adjacent to $v$, and the cardinality of $N_G(v)$, denoted by $\deg_G(v)$, is called the {\em degree} of~$v$. The minimum degree of a vertex in $G$ is denoted by $\delta(G)$. A {\em leaf\/} is a~vertex of degree one, while a {\em support vertex} (or {\em support}, for short) is a vertex adjacent to a leaf. A {\em weak support} is a vertex adjacent to exactly one leaf. The set of leaves and supports of a graph $G$ is denoted by $L_G$ and $S_G$, respectively. The {\em closed neighborhood\/} of $v$, denoted by $N_{G}[v]$, is the set $N_{G}(v)\cup \{v\}$.  In general, the {\em  neighborhood\/} of $X$, denoted by $N_{G}(X)$, is defined to be $\bigcup_{v\in X}N_{G}(v)$, and the {\em closed\/} neighborhood of $X$, denoted by $N_{G}[X]$, is the set $N_{G}(X)\cup X$. If $X$ is a set of vertices of a graph $G$ and $x\in X$, then the {\em private neighborhood} of $x$ with respect to $X$ is the set $PN_G[x,X]=N_G[x]- N_G[X-\{x\}]$, and each vertex in $PN_G[x,X]$ is called a~{\em private neighbor\/} of $x$ with respect to $X$. Finally, the {\em distance} between two vertices $u$ and $v$ in $G$ is denoted by $d_G(u,v)$, and the {\em diameter} of $G$ is ${\rm diam}(G)=\max \{ d_G(u,v) \colon  u,v \in G\}$.

A subset $D$ of $V_G$ is a {\em dominating set} of a graph~$G$ if each vertex belonging to the set $V_G -D$ has a~neighbor in $D$. The cardinality of a~minimum dominating set of $G$ is called the~{\em domination number} of $G$ and is denoted by $\gamma(G)$. Every minimum dominating set of $G$ is called a {\em $\gamma$-set} of $G$. A subset $C \subseteq V_G$ is a~{\em covering set} of $G$ if each edge of $G$ has an end-vertex in $C$. The cardinality of a~minimum covering set of $G$ is called the~{\em covering number} of $G$ and is denoted by $\beta(G)$. Any minimum covering set of $G$ is called a {\em $\beta$-set} of~$G$. Finally, a subset $I \subseteq V_G$ is said to be {\em independent\/} in $G$ if no two vertices in it are adjacent. The cardinality of a~maximum independent set of $G$ is called the~{\em independence number} of $G$ and is denoted by $\alpha(G)$. A maximum independent set in $G$ is called an {\em $\alpha$-set} of $G$. The set of all connected graphs $G$ for which $\gamma(G)=\beta(G)$ is denoted by $\cC_{\gamma=\beta}$, while $\cB$ denotes the set of all connected bipartite graphs in which the domination number is equal to the cardinality of the smaller partite set. It is easy to observe that
the complete graph $K_n$ is in $\cC_{\gamma=\beta}$ if and only if $n=2$,
the cycle $C_n$ is in $\cC_{\gamma=\beta}$ if and only if $n=4$,
the path $P_n$ is in $\cC_{\gamma=\beta}$ if and only if $n\in \{2, 3, 4, 5, 7\}$, while
the complete bipartite graph $K_{m,n}$ is in $\cC_{\gamma=\beta}$ if and only if
$\min\{m,n\}\in \{1, 2\}$.

The problem of characterizing the set $\cC_{\gamma=\beta}$  was posed by Laskar and Walikar~\cite{LW81}. Volkmann in~\cite{V94} have characterized some subsets of $\cC_{\gamma=\beta}$. A first complete characterization of the set $\cC_{\gamma=\beta}$ was given by Rall and Hartnell~\cite{HR95}, and independently by Randerath and Volkmann~\cite{RV98}. A simpler characterization was then provided by Wu and Yu~\cite{WY12}, and eventually Arumugam et al.~\cite{AJBT13} proposed another yet characterization, also studying the problem for hypergraphs. Another subset of $\cC_{\gamma=\beta}$, the set of all connected graphs $G$ in which $\gamma(H) = \beta(H)$ for every non-trivial connected induced subgraph $H$ of $G$, was characterized in \cite{AJBT13,DLSZ16} (see also \cite{TZ}).

In this paper, in Theorem~\ref{thm:gamma_B-nowa-wersja}, we provide an alternative characterization of the set~${\cal C}_{\gamma=\beta}$ in terms of $\alpha$-sets. Since $\alpha$-sets and $\beta$-sets are related by the Gallai's theorem~\cite{G59}, our characterization is natural (with respect to relations between $\gamma, \beta$, and $\alpha$), however, from the algorithmic point of view, our characterization is less practical than, for example, that in~\cite{AJBT13}, which allows to recognize whether a~graph $G$ belongs to $\cC_{\gamma=\beta}$  in $O(\sum_{v \in V_G} \deg^2(v))$ time, while ours --- does not. Next, in Theorem~\ref{thm:main2}, we provide an alternative characterization of the graphs belonging to the set $\cB$. Then, in Theorem~\ref{thm:T<=Tgamma}, we provide a constructive characterization of all the trees in the set $\cB$. Next, we discuss a quadratic time algorithm for recognizing bipartite graphs belonging to the set $\cB$,  and then, based upon an analogous argument, we conclude that the algorithm of Arumugam et al.~\cite{AJBT13} recognizes all the graphs belonging to the set $\cC_{\gamma=\beta}$ in quadratic time either. Finally, we consider the related problem of patrolling grids with mobile guards, and show that this problem can be solved in $O(n \log n + m)$ time, where $n$ is the number of line segments of the input grid and $m$ is the number of its intersection points. We emphasize that our proof techniques are similar to those in~\cite{AJBT13,HR95,RV98,WY12}, and so all the aforementioned non-constructive characterizations possess similarities.

\section{Alternative characterization of the set $\cC_{\gamma=\beta}$}\label{sec:alt}

In our characterizations of the graphs belonging to the set $\cC_{\gamma=\beta}$ or $\cB$ we use the following Gallai's theorem which relates the cardinality of the largest independent set and the cardinality of the smallest covering set in a graph.

\begin{lem}{\rm\cite{G59}} \label{Gallai} If $G$ is a graph, then $\alpha(G)+\beta(G)=|V_G|$. \end{lem}

We now present a characterization of the graphs belonging to the set $\cC_{\gamma=\beta}$ and give a~self-contained proof of this characterization.

\begin{thm}\label{thm:gamma_B-nowa-wersja}  Let $G$ be a connected graph of order at least two, and let $I$ be an $\alpha$-set of~$G$. Then $\gamma(G)=\beta(G)$ if and only if the following conditions are satisfied:
    \begin{itemize}
    \item[$(1)$] Each support vertex of $G$ belonging to $I$ is a weak support and each of its non-leaf neighbors is a support.
	   \item[$(2)$] If $vu$ is an edge of the graph $G-I$, then both vertices $v$ and $u$ are supports in $G$.
		\item[$(3)$] If $x$ and $y$ are vertices belonging to $V_G-(I \cup L_G\cup S_G)$ and $d_G(x,y)=2$, then there are at least two vertices $\overline{x}$ and $\overline{y}$ in $I$ such that $N_G(\overline{x})= N_G(\overline{y})= \{x,y\}$.
    \end{itemize}
\end{thm}

\begin{proof} Assume that $\gamma(G)=\beta(G)$.  Then, since $I$ is an $\alpha$-set of $G$ and $G$ has no isolated vertex, $J=V_G-I$ is a dominating set of $G$, and therefore $\gamma(G)\le |J|$. In addition, $\beta(G)=\gamma(G)\le |J|=|V_G|-|I|= |V_G|-\alpha(G)=\beta(G)$ (as $\alpha(G)+\beta(G)=|V_G|$ by Lemma \ref{Gallai}). Thus $\gamma(G)=|J|$ and $J$ is a~$\gamma$-set of $G$. To prove the condition (1), consider a vertex $v$ belonging to $S_G\cap I$. Then $|N_G(v)\cap L_G|=1$ (otherwise $|N_G(v)\cap L_G|\ge 2$ and $J'= (J-(N_G(v)\cap L_G))\cup \{v\}$ would be a dominating set of $G$, which is impossible as $|J'|<|J|=\gamma(G)$). Thus, let $v'$ be the only element of $N_G(v)\cap L_G$. It remains to prove that $N_G(v)\subseteq L_G\cup S_G$. Suppose to the contrary that $N_G(v)-(L_G\cup S_G)\not=\emptyset$ and consider a vertex $x\in N_G(v)-(L_G\cup S_G)$. Then $N_G(x)\cap L_G=\emptyset$ and we now claim that $J''= (J-\{v',x\})\cup \{v\}$ is a dominating set of $G$. To prove this, it suffices to show that every vertex $y \in V_G-J''$  has a neighbor in $J''$. This is obvious if $y\in \{v', x\}\cup (I-N_G(x))$. Thus assume that $y\in N_G(x)\cap I$. In this case we have $N_G(y)\cap(J-\{v',x\})\not=\emptyset$ (as $N_G(y)\subseteq J-\{v'\}$ and $|N_G(y)|\ge 2$) and therefore $N_G(y)\cap J''\not=\emptyset$. Consequently, $J''$ is a~dominating set of $G$, but this contradicts the minimality of $J$ as $|J''|<|J|= \gamma(G)$. This completes the proof of (1).

To prove (2), let us consider an edge $vu$ of $G-I$. From the fact that $I$ is an $\alpha$-set of $G$ it follows immediately that  $\{v, u\} \cap L_G=\emptyset$. It remains to prove that $\{v, u\} \subseteq S_G$. Suppose to the contrary that $v\not \in S_G$ or $u\not \in S_G$, say $v\not \in S_G$. In this case we claim that $J-\{v\}$ is a dominating set of $G$. To observe this, it suffices to show that $N_G(y)\cap (J-\{v\})\not=\emptyset$ if $y\in V_G-(J-\{v\})= \{v\}\cup (I-N_G(v))\cup (N_G(v)\cap I)$. The statement is obvious if $y=v$, as $u\in N_G(v)\cap (J-\{v\})$. Thus assume that $y\in I-N_G(v)$. In this case $y\in I= V_G-J$ and $y\not\in N_G(v)$, and therefore, $N_G(y)\cap J\not=\emptyset$ (since $J$ is a~dominating set of $G$) and $v\not\in N_G(y)$. Consequently, $N_G(y)\cap (J-\{v\}) \not=\emptyset$. Finally assume that $y\in N_G(v)\cap I$. Then $N_G(y) \subseteq J$ (since $I$ is independent), $|N_G(y)|\ge 2$ (as $y\not\in L_G$), and consequently, $N_G(y)\cap (J-\{v\}) \not=\emptyset$. This proves that $J-\{v\}$ is a~dominating set of $G$, contrary to the minimality of $J$. This finishes the proof of (2).

To prove the condition (3), let $x$ and $y$ be vertices belonging to $J-(L_G\cup S_G)$ and such that $\deg_G(x,y)=2$. Now, if $b\in N_G(x)\cap N_G(y)$, then, since $|(J-\{x,y\})\cup \{b\}|<|J|$, the set $(J-\{x,y\})\cup \{b\}$ does not dominate $G$. Therefore there exists $\overline{x} \in I-\{b\}$ for which $N_G(\overline{x})\subseteq \{x,y\}$. Thus, since $\deg_G(\overline{x})\ge 2$ (as $\{x, y\}\cap (L_G\cup S_G)=\emptyset$), $N_G(\overline{x}) =\{x,y\}$ and $\overline{x}\in I$ (otherwise, if $\overline{x}\in J$ then $J-\{\overline{x}\}$ would be a smaller dominating set of $G$). Similarly, since $(J-\{x,y\})\cup \{\overline{x}\}$ does not dominate $G$, there exists $\overline{y} \in I- \{\overline{x}\}$ for which $N_G(\overline{y})= \{x,y\}$. This proves the condition (3).

Assume now that the conditions (1)--(3) are satisfied. We claim that $\gamma(G)=|J|$, where $J=V_G-I$. Suppose to the contrary that $\gamma(G)<|J|$.  Let $D$ be a $\gamma$-set of $G$ with $|D\cap J|$ as large as possible. We get a contradiction in the three possible cases: (1) $D\varsubsetneq J$, (2) $D\subseteq I$, (3) $D\cap J\not=\emptyset$ and $D\cap I\not= \emptyset$.

{\em Case $1$}. If $D\varsubsetneq J$, then $J-D\not=\emptyset$ and for a vertex $v\in J-D$ there exists $u\in N_G(v)\cap D$. Then $vu$ is an edge in $G-I$ and therefore $v\in S_G$ (by the condition (2)). Thus the set $N_G(v)\cap L_G$ is nonempty. Moreover $N_G(v)\cap L_G\subseteq D$ (since every element of $N_G(v)\cap L_G$ has to be dominated) and $N_G(v)\cap L_G\subseteq I$ (by the choice of $I$), a contradiction to $D\varsubsetneq J$.

{\em Case $2$}. If $D\subseteq I$, then $D=I$ (as $I$ is independent and no proper subset of $I$ dominates all vertices in $I$). The set $J$ also is independent as otherwise the set $S_G\cap J$ would be non-empty (by the condition (2)) and it would be a subset of $D$ (by the choice of $D$), which is impossible (as $D$ and $J$ are disjoint). Consequently, $G$ is a bipartite graph and the sets $I=D$ and $J$ form a bipartition of $V_G$ into independent sets, and $|I|= |D|= \gamma(G)< |J|\le \alpha(G)=|I|$, a contradiction.

{\em Case $3$}. Finally assume that $D\cap J\not=\emptyset$ and $D\cap I\not= \emptyset$. In this case, from the supposition $|J|>\gamma(G)= |D|$, it follows that $|J-D|>|I\cap D|\ge 1$.  Now the choice of $D$, the maximality of $I$, and the condition (2) imply that each vertex belonging to $J-D$ has a~neighbor in $I\cap D$. Therefore the pigeonhole principle implies that there are two vertices $x$ and $y$ in $J-D$  which are adjacent to the same vertex in $I\cap D$. Since $N_G(I\cap L_G) \subset D$ (by the choice of $D$) as well as $J \cap L_G \subset D$ (by (1) and the choice of $D$), the vertices $x$ and $y$ belong to $J-(L_G\cup S_G)$. Therefore, by (3), there exist vertices $\overline{x}$ and $\overline{y}$ in $I$ for which $N_G(\overline{x})= N_G(\overline{y})=\{x,y\}$. Furthermore, the vertices $\overline{x}$ and $\overline{y}$ belong to $I\cap D$ (as otherwise $\overline{x}$ and $\overline{y}$ would not be dominated by $D$). Now the set $D'= (D-\{\overline{x},\overline{y}\})\cup \{x, y\}$ is a~dominating set of $G$, which is impossible as $|D'|=|D|$ and $|D'\cap J|>|D\cap J|$. This proves that $\gamma(G)=|J|$ and implies that $\gamma(G)=\beta(G)$ as $|J|=|V_G|-|I|= |V_G|-\alpha(G)= \beta(G)$ (by Lemma \ref{Gallai}).
\end{proof}

From  Theorem \ref{thm:gamma_B-nowa-wersja} we have the following two immediate corollaries.

\begin{cor} \label{cor:delta2}{\rm \cite{V96}} If a graph $G$ belongs to the set $\cC_{\gamma=\beta}$, then $\delta(G) \le 2$. \end{cor}

\begin{cor} \label{cor:gamma-beta-delta2}{\rm \cite{HR95,RV98}} If a graph $G$ belongs to the set $\cC_{\gamma=\beta}$ and $\delta(G)=2$, then $G$ is a~bipartite graph. \end{cor}

Fig.~\ref{rys-gamma=beta} shows a graph $G$ belonging to the set $\cC_{\gamma=\beta}$. In this case the solid vertices form an $\alpha$-set $I$ of $G$, while $V_G-I$ is a $\gamma$- and $\beta$-set of $G$. Certainly, $I$ satisfies the conditions (1)--(3) of Theorem \ref{thm:gamma_B-nowa-wersja}. On the other hand the graph $F$ shown in Fig.~\ref{rys-gamma=beta} does not belong to the set $\cC_{\gamma=\beta}$, as $\gamma(F)=2$, while $\beta(F)=3$. Thus, no $\alpha$-set $I$ of $F$ satisfies all three conditions (1)--(3) of Theorem \ref{thm:gamma_B-nowa-wersja}. It is easy to check that the sets $I_1 =\{v_2, v_4, v_6\}$, $I_2 =\{v_1, v_4, v_6\}$, $I_3 =\{v_1, v_3, v_5\}$ are the only $\alpha$-sets of $F$, and, in addition, the $\alpha$-set $I_k$ ($k\in\{1, 2, 3\}$) satisfies precisely the conditions $\{(1), (2), (3)\}-\{(k)\}$ of Theorem~\ref{thm:gamma_B-nowa-wersja}.

\begin{figure}[t!] \begin{center}
{\epsfxsize=4.5in \epsffile{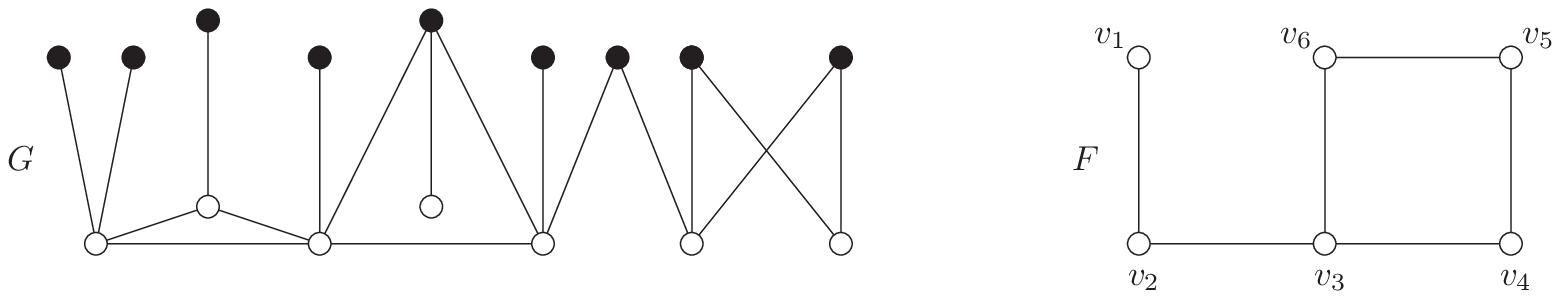}} \vspace{-4mm}
\caption{Graphs $G$ and $F$, where $G \in \cC_{\gamma=\beta}$, while $F \not\in \cC_{\gamma=\beta}$.} \label{rys-gamma=beta}
\end{center}\end{figure}

\section{Bipartite graphs with the largest domination number}\label{sec:trees}
It is obvious that if $G = ((A,B),E_G)$ is a bipartite graph, then each of the sets $A$ and $B$ is dominating in $G$ and therefore $\gamma(G) \le \min\{|A|,|B|\}$. Here we study graphs $G = ((A,B),E_G)$ for which the equality $\gamma(G) = \min\{|A|,|B|\}$ holds, that is, we study graphs belonging to the set $\cB$ of bipartite graphs in which the domination number is equal to the cardinality of the smaller partite set. Such graphs were studied in \cite{HR95} and~\cite{RV98}. Our characterization given in the next theorem is similar but different from those in~\cite{HR95,RV98} (see Theorems 3.6 and 4.1 in \cite{HR95}, and Theorems 3.3 and 3.4  in~\cite{RV98}). Recently Miotk et al.~\cite{MTZxx} observed that the graphs belonging to the set $\cB$ can also be characterized in terms of some graph operations.


\begin{thm} \label{thm:main2} Let $G=((A,B),E_G)$ be a connected bipartite graph with $1\le |A|\le |B|$. Then the following statements are equivalent:
\begin{itemize}
\item[$(1)$] $\gamma(G)=|A|$.
\item[$(2)$] $\gamma(G)=\beta(G)=|A|$.
\item[$(3)$] $G$ has the following two properties:
\begin{itemize}
\item[{\rm (a)}] Each support vertex of $G$ belonging to $B$ is a weak support and each of its non-leaf neighbors is a support.
\item[{\rm (b)}]
If $x$ and $y$ are vertices belonging to $A-(L_G\cup S_G)$ and $d_G(x,y)=2$, then there are at least two vertices $\overline{x}$ and $\overline{y}$ in $B$ such that $N_G(\overline{x})= N_G(\overline{y})= \{x,y\}$.
\end{itemize}
\end{itemize}
\end{thm}

\begin{proof}  Assume first that  $\gamma(G)=|A|$. We claim that then $\alpha(G)=|B|$ and $\beta(G)=|A|$. It is obvious that $\alpha(G)\ge |B|$ (as $B$ is independent in $G$). Thus, it suffices to prove that $\alpha(G)\le |B|$. Suppose to the contrary that $\alpha(G)> |B|$, and let $I$ be an $\alpha$-set of $G$. Then $|I|=\alpha(G)>|B|$,  $V_G-I$ is a dominating set of $G$, and therefore $\gamma(G)\le |V_G-I|= |V_G|-|I| <|V_G|-|B|= |A|= \gamma(G)$, a contradiction. Thus,  $\alpha(G)=|B|$ and $B$ is an $\alpha$-set of $G$. From the equality $\alpha(G)=|B|$ and from Lemma~\ref{Gallai} it follows that $\beta(G)=|A|$. Consequently, (2) $\gamma(G)=\beta(G)=|A|$. Now, since $\gamma(G)= \beta(G)$ and $B$ is an $\alpha$-set of $G$, Theorem \ref{thm:gamma_B-nowa-wersja} implies that $G$ has the properties (a) and~(b) of the statement (3).

Assume now that $G$ has the properties (a) and~(b) of the statement (3). In a standard way (as in~\cite{HR95,RV98}), we prove that $\gamma(G)=|A|$. Since $A$ is a dominating set of $G$, we have $\gamma(G)\le |A|$, and, therefore, it suffices to show that $\gamma(G)\ge |A|$. Suppose to the contrary that $\gamma(G)<|A|$. Let $D$ be a~$\gamma$-set of $G$ with $|D\cap A|$ as large as possible. Since $|A-D|>|D\cap B|\ge 1$ and since each vertex in $A-D$ has a~neighbor in $D\cap B$, the pigeonhole principle implies that there are two vertices $x$ and $y$ in $A-D$  which are adjacent to the same vertex in $D\cap B$. Since $N_G(B\cap L_G) \subset D$ (by the choice of $D$) as well as $A \cap L_G \subset D$ (by (a) and the choice of $D$), the vertices $x$ and $y$ belong to $A-(L_G \cup S_G)$. Therefore, by (b), there exist vertices $\overline{x}$ and $\overline{y}$ in $B$ for which $N_G(\overline{x})= N_G(\overline{y})=\{x,y\}$. The vertices $\overline{x}$ and $\overline{y}$ belong to $D\cap B$ (as otherwise $\overline{x}$ and $\overline{y}$ would not be dominated by~$D$). Now, the set $D'= (D-\{\overline{x},\overline{y}\})\cup \{x, y\}$ is a dominating set of $G$, which is impossible as $|D'|=|D|$ and $|D'\cap A|>|D\cap A|$. \end{proof}

As an immediate consequence of Theorem \ref{thm:main2} and its proof, we have the following results.

\begin{cor} \label{wniosek-alpha-beta}
Let $G=((A,B),E_G)$ be a connected bipartite graph with $1\le |A|\le |B|$. If $\gamma(G)=|A|$, then $\alpha(G)=|B|$ and $\beta(G)=|A|$.
\end{cor}

\begin{cor} \label{wniosek-B-zawarte-w-C} The set $\cB$ is a subset of the set $\cC_{\gamma=\beta}$. \end{cor}

\begin{cor} \label{wniosek5dladrzew} Let $T$ be a tree. If $(A,B)$ is a bipartition of $T$ and $1\le |A|\le |B|$, then the following statements are equivalent:
\begin{itemize}
\item[$(1)$] $\gamma(T)=|A|$.
 \item[$(2)$] $\gamma(T)=\beta(T)=|A|$.
\item[$(3)$] $T$ has the following two properties:
\begin{itemize}
\item[{\rm (a)}] If $v\in B\cap S_T$, then $v$ is a weak support and $N_T(v)-L_T\subseteq S_T$.
\item[{\rm (b)}] If $x$ and $y$ are vertices belonging to $A-(L_T\cup S_T)$, then $d_T(x,y)\not=2$ $($or, equivalently,  $|N_T(z) \cap (A-S_T)|\le 1$ for every $z$ belonging to $B-(L_T\cup S_T)$$)$.
    \end{itemize} \end{itemize} \end{cor}

The {\em corona\/} $F \circ K_1$ of a graph $F$ is the graph formed from $F$ by adding a new vertex $v'$ and edge $vv'$ for each vertex $v$ of $F$. A graph $G$ is said to be a {\em corona graph\/} if $G = F \circ K_1$ for some graph $F$. We note that a graph $G$ is a corona graph if and only if each vertex of $G$ is a leaf or it is adjacent to exactly one leaf of $G$, and consequently every corona graph belongs to the set $\cC_{\gamma=\beta}$. Payan and Xuong \cite{PayanXuong}, and Fink et al. \cite{FJKR85} have proved  that for a connected graph $G$ of even order is $\gamma(G)=|V_G|/2$ if and only if $G$ is the cycle $C_4$ or the corona $F\circ K_1$ for any connected graph $F$ (see also~\cite{TV} for a short proof). From this (or directly from Theorem \ref{thm:main2}) we immediately have the following corollary.

\begin{cor}\label{prop:corona_gamma}
Let $G=((A,B),E_G)$ be a connected bipartite graph with $|A| = |B|$. Then $\gamma(G)=|A|$ if and only if $G$ is the cycle $C_4$ or the corona of a connected bipartite graph.
\end{cor}

Simple examples illustrating relations between graphs considered in Theorems~\ref{thm:gamma_B-nowa-wersja} and~\ref{thm:main2}, and Corollaries~\ref{cor:delta2},~\ref{cor:gamma-beta-delta2}, and~\ref{wniosek-alpha-beta}
--\ref{prop:corona_gamma} are shown in Fig.~\ref{fig:examples}.

\begin{figure}[h!]
\begin{center}
{\epsfxsize=5.5in \epsffile{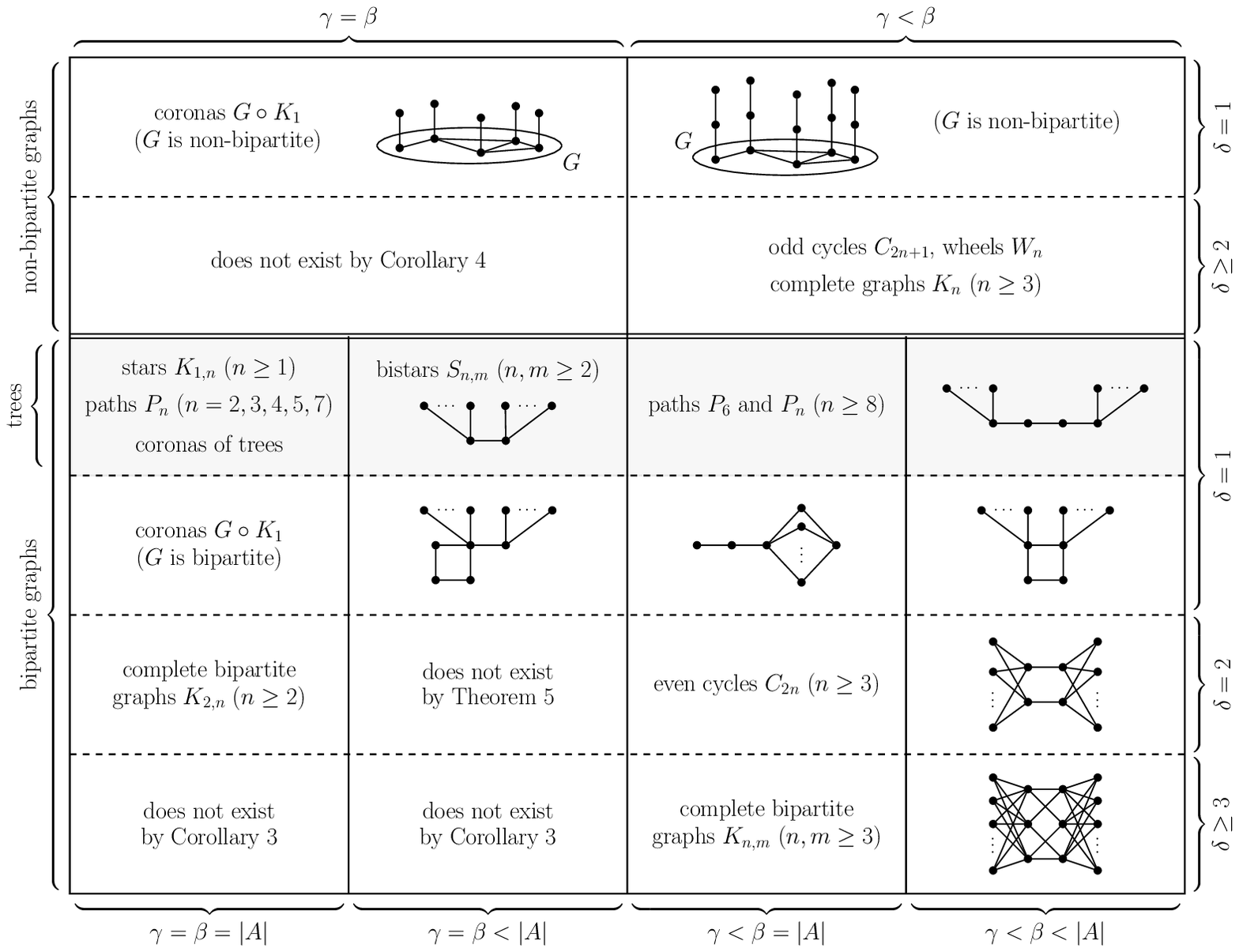}}
\vspace{-2mm}
\caption{}\label{fig:examples}
\end{center}
\end{figure}

Let  ${\cal T}_{\rm max}$ denote the set of trees for which the domination number is equal to the size of its smaller partite set. We now provide a constructive characterization of the trees belonging to the set ${\cal T}_{\rm max}$. Similar constructive characterizations of trees for different domination related parameters or properties have been presented in e.g.~\cite{ADR16,HK17,HS16,R17}. Our characterization is based on four simple operations. To present one of these operations, we introduce some additional term. A vertex $v$ of a graph $G$ is called {\em $\gamma^-$-critical} if $\gamma(G-v) < \gamma(G)$ (or, equivalently, if $\gamma(G-v) = \gamma(G)-1$). Such vertices have been intensively studied (see e.g. \cite{Haynes-Henning,Samodivkin,Sampath-Neeralagi}). In particular, Sampathkumar and Neeralagi \cite{Sampath-Neeralagi} have observed that a vertex $v$ of $G$ is $\gamma^-$-critical if and only if $PN_G[v,D]=\{v\}$ for some $\gamma$-set $D$ containing $v$.

Let ${\cal T}$ be the family of trees $T$ that can be obtained from a sequence of trees $T_0, \ldots, T_k$, where $k\ge 0$, $T_0=K_2$ and $T=T_k$. In addition, if $k\ge 1$, then for each $i=1,\ldots,k$, the tree $T_i$ can be obtained from the tree $T'=T_{i-1}$ with the bipartition $(A',B')$ by one of the following four operations ${\cal O}_1$, ${\cal O}_2$, ${\cal O}_3$, and ${\cal O}_4$ defined below and illustrated in Fig.~\ref{operacjeO1-O4}.

\begin{itemize} \item[${\cal O}_1$] Add a new vertex $b$ to $T'$ and join $b$ to an $A'$-vertex $a'$ of $T'$.
\item[${\cal O}_2$] Add a new vertex $a$ to $T'$ and join $a$ to a $B'$-vertex $b'$ of $T'$ such that $N_{T'}(b')\subseteq S_{T'}$ and $b'\not\in L_{T'}$.
\item[${\cal O}_3$] Add two new vertices $a$ and $b$ to $T'$ and join $b$ to $a$ and to an $A'$-vertex $a'$ of $T'$ that is a support of $T'$.
\item[${\cal O}_4$] Add two new vertices $a$ and $b$ to $T'$ and join $a$ to $b$ and to a $B'$-vertex $b'$ of $T'$ which is not a $\gamma^-$-critical vertex in $T'$.
\end{itemize}

\begin{figure*}[!h]
\begin{center}
{\epsfxsize=5in \epsffile{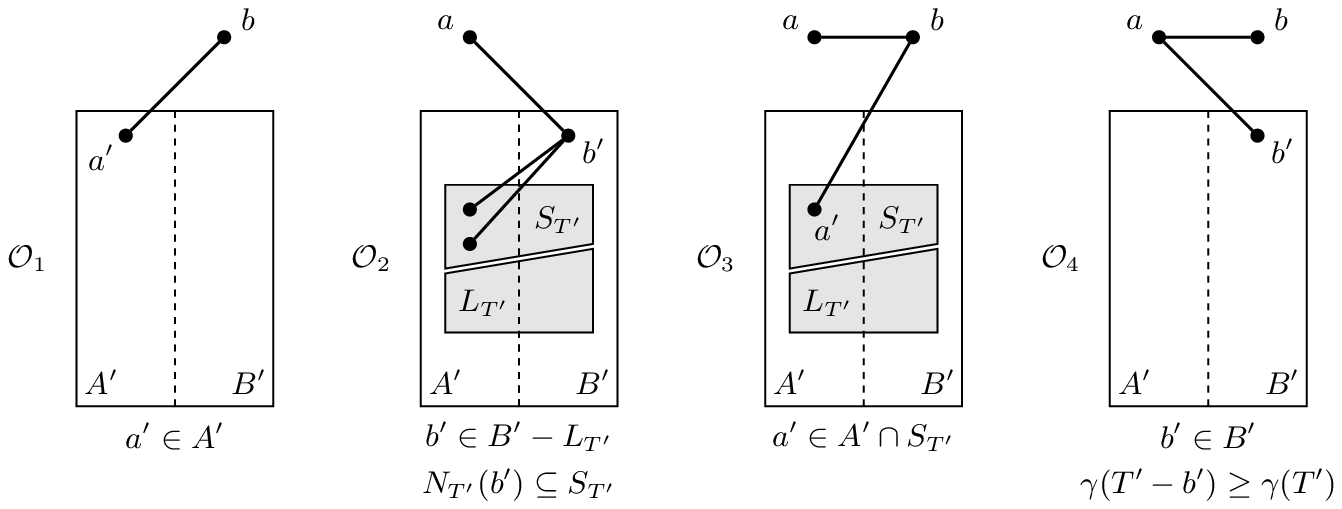}}

\vspace{-3mm}
\caption{}\label{operacjeO1-O4}
\end{center}
\end{figure*}

Each of the vertices $a'$ in the above defined operations ${\cal O}_1$ and ${\cal O}_3$, and $b'$ in the operations ${\cal O}_2$ and ${\cal O}_4$, is called the {\em attacher\/} of $T'$. We shall prove that the trees in ${\cal T}$ are precisely the trees belonging to the family ${\cal T}_{\rm max}$. Our first aim is to show that each tree in ${\cal T}$ belongs to the family ${\cal T}_{\rm max}$. For this purpose we prove the following lemma.

\begin{lem}\label{thm:Tgamma<=T}
If $T \in {\cal T}$, then $T \in {\cal T}_{\rm max}$. \end{lem}

\begin{proof} Let $T_0, T_1, T_2,\ldots$ be a sequence of trees, where $T_0=K_2$ (with an arbitrary chosen bipartition $(A',B')$ of its 2-element vertex set) and every $T_n$ can be obtained from the tree $T_{n-1}$ by one of the operations  ${\cal O}_1$, ${\cal O}_2$, ${\cal O}_3$, and ${\cal O}_4$. By induction on $n$ we shall prove that $T_n\in {\cal T}_{\rm max}$ for every $n\in \mathbb{N}$. If $n=0$, then $T_0=K_2$ and
$T_0\in {\cal T}_{\rm max}$ (as $\gamma(K_2)=\beta(K_2)=1$). Assume that $n\ge 1$ and $T_{n-1}=((A',B'),E)$ is a tree, where $(A',B')$ is a bipartition of $T_{n-1}$ for which $|A'|\le |B'|$, and $\gamma(T_{n-1}) = \beta(T_{n-1})=|A'|$. We consider four cases depending on which operation is used to construct the tree $T=T_n$ from $T'=T_{n-1}$.

{\em Case $1$}. {\em $T$ is obtained from $T'$ by Operation ${\cal O}_1$}. In this case $(A',B'\cup\{b\})$ is the bipartition of $T$ and $|A'|<|B'\cup\{b\}|$. Now, since $A'$ is a $\gamma$-set of $T'$ and $b$ is adjacent to the attacher $a'$ belonging to $A'$,  $A'$ is a dominating set of $T$ and therefore $\gamma(T)\le |A'|=\gamma(T')$. Consequently,
$\gamma(T) =|A'|$ (for if it were $\gamma(T)< |A'|$ and if $D$ were a~$\gamma$-set of $T$, then $(D-\{b\})\cup\{a'\}$ would be a dominating set of $T'$ and it would be $\gamma(T')\le |(D-\{b\})\cup\{a'\}|=|D|<|A'| = \gamma(T')$, which is a contradiction). This proves that $T\in {\cal T}_{\rm max}$.

{\em Case $2$}. {\em $T$ is obtained from $T'$ by Operation ${\cal O}_2$}. Now $\gamma(T')= |A'|<|B'|$ (for if it were $\gamma(T')= |A'|=|B'|$, then $T'$ would be the corona of a tree (see Corollary \ref{prop:corona_gamma}) and the operation ${\cal O}_2$ could not be applied to $T'$) and $(A'\cup\{a\},B')$ is the bi\-par\-ti\-tion of $T$. We now claim that $\gamma(T)= |A'\cup \{a\}|=|A'|+1$. The inequality $\gamma(T)\le |A'|+1$ is obvious as $A'\cup \{a\}$ is a~dominating set of $T$. Thus, it remains to prove that $\gamma(T)\ge |A'|+1$. Suppose to the contrary that $\gamma(T) < |A'|+1$. Let $D$ be a $\gamma$-set of $T$. Since $T$ is a connected graph of order at least three, it is obvious that we may assume that no leaf of $T$ belongs to $D$. Then $b'\in D$ and, since $N_T(b')-\{a\}= N_{T'}(b') \subseteq S_{T'}= S_T- \{a\} \subseteq D$, $a$ is the only private neighbor of $b'$ with respect to $D$ in $T$, that is, $PN_T[b',D]=\{a\}$. This implies that $D-\{b'\}$ is a dominating set of $T'$. But then we have $\gamma(T') \le |D-\{b'\}|= \gamma(T)-1< |A'|= \gamma(T')$, a contradiction. This proves that $T\in {\cal T}_{\rm max}$.

{\em Case $3$}. {\em $T$ is obtained from $T'$ by Operation ${\cal O}_3$}. This time $(A'\cup\{a\},B'\cup\{b\})$ is the bipartition of $T$ and, certainly, $\gamma(T)\le |A'\cup \{a\}|=|A'|+1$. It remains to prove that $\gamma(T)\ge |A'|+1$. Suppose to the contrary that $\gamma(T) < |A'|+1$. Let $D$ be a $\gamma$-set of $T$. Since $T$ is a connected graph of order at least three, it is obvious that we may assume that no leaf of $T$ belongs to $D$. This implies that $b$ and the attacher $a'$ are in $D$. Now it is easy to observe that $D-\{b\}$ is a dominating set of $T'$ and therefore $\gamma(T') \le |D-\{b\}|= \gamma(T)-1< |A'|= \gamma(T')$, a contradiction. This proves that $T\in {\cal T}_{\rm max}$.

{\em Case $4$}. {\em $T$ is obtained from $T'$ by Operation ${\cal O}_4$}. In this case $(A'\cup\{a\},B'\cup\{b\})$ is the bipartition of $T$ and $|A'\cup\{a\}|\le |B'\cup\{b\}|$. Therefore $\gamma(T)\le |A'\cup\{a\}| =|A'|+1$. We claim that $\gamma(T)= |A'\cup \{a\}|=|A'|+1$. It remains to prove that $\gamma(T)\ge |A'|+1$. Suppose to the contrary that $\gamma(T) < |A'|+1$. Let $D$ be a $\gamma$-set of $T$. Then $\gamma(T)=|D|\le |A'|=\gamma(T')$. We may assume that $a\in D$ (otherwise we could replace $D$ with $(D-\{b\})\cup\{a\}$). Now, the attacher $b'$ is dominated only by $a$ (as otherwise $D-\{a\}$ would be a dominating set of $T'$ and it would be $\gamma(T')\le |D-\{a\}|<|A'|=\gamma(T')$, a contradiction). Then $D-\{a\}$ is a dominating set of $T'-b'$. Consequently, $\gamma(T'- b')\le |D-\{a\}|<|A'|=\gamma(T')$, contradicting the premise that $b'$ is not $\gamma^-$-critical in $T'$, that is, $\gamma(T'- b')\ge \gamma(T')$ for the attacher $b'$ in Operation~${\cal O}_4$. This proves that $T\in {\cal T}_{\rm max}$ and completes the proof. \end{proof}

We are now ready to provide a constructive characterization of the trees belonging to the family ${\cal T}_{\rm max}$.

\begin{thm}\label{thm:T<=Tgamma} A tree $T$ belongs to the family ${\cal T}_{\rm max}$ if and only if $T$ belongs to the family~${\cal T}$.
\end{thm}

\begin{proof} It follows from Lemma \ref{thm:Tgamma<=T} that ${\cal T}\subseteq {\cal T}_{\rm max}$. Conversely, suppose $T$ is a tree belonging to the family ${\cal T}_{\rm max}$. Let $(A,B)$ be a bipartition of $T$, where $1\le |A|\le |B|$. By induction on the order of $T$ we shall  prove that if $\gamma(T)=|A|$, then $T$ belongs to ${\cal T}$, that is, $T$ can be obtained from $K_2$ (which belongs to ${\cal T}$ and ${\cal T}_{\rm max}$) by repeated applications of the operations ${\cal O}_1$, ${\cal O}_2$, ${\cal O}_3$, and ${\cal O}_4$. If $T$ is a~star on at least three vertices, then $T$ can be obtained from $K_2$ by repeated applications of operation ${\cal O}_1$ (that is, attaching leaves to the only $A$-vertex of $K_2$), thus implying that $T\in {\cal T}$. Hence, we may assume
that $\mbox{\rm diam}(T)\ge 3$. We consider two cases: $|A|=|B|$, $|A|<|B|$.

{\it Case} 1. If $|A|=|B|$ and $\gamma(T)=|A|$, then it follows from Corollary   \ref{prop:corona_gamma} that $T$ is the corona of some tree $R$. Thus $T$ has a vertex, say $v$, of degree 2. Let $v'$ be the only non-leaf neighbor of $v$. Let $l$ and $l'$ be the only leaves adjacent to $v$ and $v'$, respectively. Now, it is obvious that the tree $T'=T-\{v,l\}$ is the corona (of the tree $R-v$) and therefore $\gamma(T')= |V_{T'}|/2 = |A|-1=|B|-1$. Thus, $T'\in {\cal T}_{\rm max}$ and the induction hypothesis implies that $T'\in {\cal T}$. Consequently $T\in {\cal T}$, since  $T$ can be rebuilt from $T'$ by applying the operation ${\cal O}_3$ (resp.\ ${\cal O}_4$) if the attacher $v'$ is an $A$-vertex (resp.\ a~$B$-vertex) in $T'$.

{\it Case} 2. If $|A|<|B|$, then we consider two subcases: $A\cap L_T\not= \emptyset$, $A\cap L_T= \emptyset$.

{\it Subcase} 2a. Assume first that $A\cap L_T\not= \emptyset$. Let $v$ be a vertex belonging to $A\cap L_T$, and let $v'$ be the only neighbor of $v$. Then it follows from Corollary~\ref{wniosek5dladrzew} that every vertex belonging to $N_T(v')-\{v\}$ is a support vertex in $T$. Now $(A-\{v\},B)$ is a bipartition of the subtree $T'=T-v$ and it is easy to observe that $\gamma(T')= |A-\{v\}|<|B|$ (for if it were $\gamma(T')< |A-\{v\}|$ and if $D$ were a $\gamma$-set of $T'$, then $D\cup\{v\}$ would be a dominating set of $T$ and it would be $\gamma(T)\le |D\cup\{v\}| = \gamma(T')+1<|A|=\gamma(T)$, a contradiction). Consequently, $T'\in {\cal T}_{\rm max}$ and the induction hypothesis implies that $T'\in {\cal T}$. Now, since every neighbor of $v'$ in $T'$ is a support vertex, the tree $T$ can be rebuilt from $T'$ by applying the operation ${\cal O}_2$ with the attacher $v'$. This proves that $T\in {\cal T}$.

{\it Subcase} 2b. Finally assume that $A\cap L_T= \emptyset$. Then $L_T\subseteq B$ and
$S_T\subseteq A$. Let $(x_0, x_1,\ldots, x_d)$ be the longest path in $T$. Since $x_0$ and $x_d$ are leaves in $T$ and $d=\mbox{\rm diam}(T)\ge 3$, necessarily $x_0\in B$, $x_1\in A$, $x_2\in B$, $x_3\in A$, $\ldots$, $x_d\in B$ and therefore $d\ge 4$. If $\deg_T(x_1)>2$, then $(A,B -\{x_0\})$ is a bipartition of the subtree $T'=T-x_0$ of $T$ and it is obvious that $\gamma(T')= |A|\le |B -\{x_0\}|$. Thus $T'\in {\cal T}_{\rm max}$ and the induction hypothesis implies that $T'\in {\cal T}$. Now the tree $T$ can be rebuilt from $T'$ by applying the operation ${\cal O}_1$ with the attacher $x_1$. If $\deg_T(x_1)=2$, then $(A-\{x_1\},B -\{x_0\})$ is a~bipartition of the subtree $T'=T-\{x_0,x_1\}$ of $T$ and $|A-\{x_1\}|< |B -\{x_0\}|$. It is a simple matter to see that $\gamma(T')= |A-\{x_1\}|=|A|-1$. Thus $T'\in {\cal T}_{\rm max}$ and the induction hypothesis implies that $T'\in {\cal T}$. We now claim that $x_2$ is not a $\gamma^-$-critical vertex in $T'$. Suppose, contrary to our claim, that $x_2$ is a $\gamma^-$-critical vertex in $T'$. Then $\gamma(T'-x_2)=\gamma(T')-1=|A|-2$. But now, if $D$ is a $\gamma$-set of $T'-x_2$, then $D\cup\{x_1\}$ is a~dominating set of $T$ and $\gamma(T)\le |D\cup\{x_1\}|=|A|-1<\gamma(T)$, a contradiction. Consequently, since $x_2$ is not a $\gamma^-$-critical vertex in $T'$, the tree $T$ can be rebuilt from $T'$ by applying the operation ${\cal O}_4$ with the attacher $x_2$. This completes the proof. \end{proof}


\section{Algorithmic consequences}\label{sec:alg}

Our characterization of graphs belonging to the set ${\cal C}_{\gamma=\beta}$, given in Theorem~\ref{thm:gamma_B-nowa-wersja}, is non-practical from algorithmic point of view, since this characterization involves $\alpha$-sets. However Theorem~\ref{thm:main2} allows us to propose a quadratic-time algorithm for recognizing bipartite graphs with the domination number equal to the size of the smaller partite set. The idea of our algorithm follows that in~\cite{AJBT13}, with the only difference of a slightly more thorough running time analysis, based upon the following obvious lemma.

\begin{lem}\label{lem:n2}
In a bipartite graph $G=((A,B),E_G)$ of order $n$ there are at most $n/2$ subsets $\{x,y\} \subseteq A$ such that $d_G(x,y)=2$ and for which there are at least two vertices $\overline{x}$ and $\overline{y}$ in $B$ satisfying $N_G(\overline{x})= N_G(\overline{y}) = \{x,y\}$.
\end{lem}


The next corollary is immediate from Theorem~\ref{thm:main2} and Lemma~\ref{lem:n2}.

\begin{cor}\label{cor:n2}
Let $G=((A,B),E_G)$ be a connected $n$-vertex bipartite graph with $1\le |A| \le |B|$. If $\gamma(G)=|A|$, then there are at most $n/2$ subsets $\{x,y\} \subseteq A -(L_G\cup S_G)$ for which $d_G(x,y)=2$.
\end{cor}

\begin{thm}\label{thm:algo}
If  $G=((A,B),E_G)$ is a connected $n$-vertex bipartite graph with $1\le |A| \le |B|$, then the equality $\gamma(G)=|A|$ can be verified in $O(n^2)$ time.
\end{thm}

\begin{proof}
First observe that if $|A|=|B|$, then by Corollary~\ref{prop:corona_gamma} all we need is to verify whether $G$ is the cycle $C_4$ or a corona graph, which can be done in $O(n+m)$ time, where $m=|E_G|$. Thus assume $|A| < |B|$. Then it suffices to verify the properties  (3a) and (3b) in Theorem~\ref{thm:main2}. The first property can be easily verified in $O(n^2)$ time by identifying and distinctly marking non-leaf and support vertices. To verify the second property, we first determine and distinctly mark the vertices in $A'=A - (L_G\cup S_G)$ by using the prior markings. Next,  similarly as in~\cite{AJBT13}, by considering vertices in $B$ of degree two whose both neighbors are in~$A'$, we create a multigraph $M$ (represented by the adjacency matrix, constructed in $O(n^2)$ time) on the vertex set $A'$ in which the multiplicity of each edge joining two vertices is equal to the number of their common neighbors of degree two in $B$. Then, for each vertex $b \in B$, we construct the list $L(b)$ of vertices adjacent to $b$ in $A'$. Since each of these lists is of length at most~$n$, all that can be easily done in $O(n^2)$ time.  We continue by checking for each $b \in B$ and every two vertices $a$ and $a'$ belonging to $L(b)$ whether the multiplicity of edge $aa'$ in $M$ is at least $2$, and we stop whenever checking fails (as then $\gamma(G) \neq |A|$). As observed in~\cite{AJBT13}, this may require $\Theta(\sum_{b \in B} |L(b)|^2)=\Theta(\sum_{v \in V_G} \deg^2(v))$ time. However, since a subset $\{a,a'\}$ can be checked at most $n-2$ times, whereas at most $n/2$ of such subsets can be positively examined (by Corollary~\ref{cor:n2}), the algorithm always stops after $O(n^2)$ steps.
\end{proof}

Lemma~\ref{lem:n2} allows to conclude that all the graphs belonging to the set $\cC_{\gamma=\beta}$ can also be recognized in $O(n^2)$ time. Namely,  Arumugam et al.~\cite{AJBT13} proposed an algorithm that recognizes whether a given graph $G$ belongs to the set $\cC_{\gamma=\beta}$ in the claimed time complexity of $O(\sum_{v \in V_G} \deg^2(v))$. The most time consuming step in their algorithm, after a preprocessing step for determining the relevant bipartition $(A, B)$ of the subgraph $H$ resulting from $G$ by deleting all edges whose both end-vertices are support vertices, is to check whether for every non-support distinct vertices $u,v \in A$, if $u$ and $v$ have some common neighbor, then they have at least two common neighbors of degree two, which takes  $O(\sum_{b \in B} \deg^2_H(b))=O(\sum_{v \in V_G} \deg^2(v))$ time in total. However, similarly as in the proof of Theorem~\ref{thm:algo}, notice that a $2$-element subset $\{u,v\}$ of $A$ can be checked at most $n-2$ times, whereas at most $n/2$  such subsets can `pass the test' by Lemma~\ref{lem:n2} and Theorem~1.1 in~\cite{AJBT13}. Hence, we obtain the following corollary.

\begin{cor}\label{cor:algoC}
Let $G$ be an $n$-vertex graph without isolated vertices. Then the equality $\gamma(G)=\beta(G)$ can be verified in $O(n^2)$ time.\eop
\end{cor}

It is a natural question to ask whether the running time analysis of our algorithm (and that in~\cite{AJBT13}) can be improved, say, to obtain the running time  $O(n^{2-\varepsilon})$ for some $\varepsilon >0$. Showing that there exists an infinite class of bipartite graphs $G=((A,B),E_G)$ in which the most time consuming step of our algorithm being of order $\sum_{b \in B} |L(b)|^2$ can be $\Theta(n^{2})$, we prove that this question has a negative answer. Let $n$ and $p$ be integers, where $n\ge 16$ and $p=\lfloor \sqrt{n}/2 \rfloor$. The multigraph of order $p$ in which every two vertices are joined by exactly two multiple edges is denoted by $K_p''$. Now let $G$ be the graph obtained from the join $K_p''+\overline{K}_{n-p^2}$ by subdividing each of its double edges exactly once (as illustrated in Fig.~\ref{fig-graph}, where $S(K_p'')$ is the subdivision graph of $K_p''$). It is obvious that $G$ is a leafless bipartite graph in which $A=V_{K_p''}$ and $B=V_G-A$ are partite sets and $1<|A|=p<n-p= |B|$. Because sets $L_G$ and $S_G$ are empty, and for every distinct vertices $x$ and $y$ in $A$ there are distinct vertices $\overline{x}$ and $\overline{y}$ in~$B$ such that $N_G(\overline{x})= N_G(\overline{y})= \{x,y\}$, it follows from Theorem~\ref{thm:main2} that $\gamma(G)=|A|$. On the other hand, it immediately follows from the construction of $G$ that $\sum_{v \in B} \deg^2_G(v)= \Theta(n^2)$, and so our algorithm (and that in~\cite{AJBT13}), when executed on $G$, will stop only after $\Theta(n^2)$ steps.

\begin{figure*}[!t] \begin{center}

{\epsfxsize=2in \epsffile{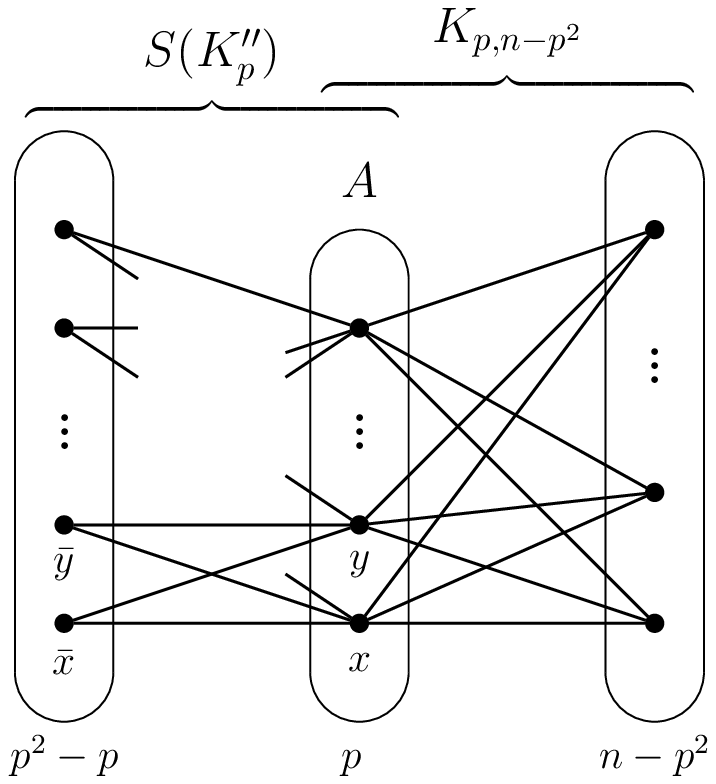}}

\vspace{-2mm}
\caption{Bipartite graph $G$ with $\sum_{v \in V_G-A} \deg^2_G(v)=\Theta(n^2)$.}\label{fig-graph}
\end{center}\end{figure*}

Finally, Theorem~\ref{thm:algo} itself has also a practical application in guarding grids. Let ${\cS}=\{S_1,S_2,\ldots,S_n \}$ be a family of distinct vertical and horizontal closed line segments in the plane, where every two collinear line segments are disjoint. This family is called a {\em grid\/} if $n\ge 2$ and the union $\bigcup{\cS}= S_1 \cup S_2 \cup \cdots \cup S_n$ is a~connected subset of $\mathbb{R}^2$. The intersection graph of the family ${\cS}$ is denoted by $G_{\cS}$ (and it is the graph with vertex-set ${\cS}$ and edge-set $\{S_iS_j\colon i\not=j, \,\,S_i, S_j\in {\cS},\,\, \mbox{and} \,\, S_i\cap S_j\not=\emptyset\}$). It is obvious that if ${\cS}$ is a grid, then $G_{\cS}$ is a connected bipartite graph and the pair $(V_{\cS},H_{\cS})$ is its bipartition, where $V_{\cS}$ ($H_{\cS}$, resp.) is the set of all vertical (horizontal, resp.) line segments of ${\cS}$~\cite{Nta86}. An example of a~grid $\cS$ and the intersection graph $G_{\cS}$ (corresponding to the family ${\cS}= V_{\cS}\cup H_{\cS}$, where $V_{\cS}=\{x, y,z, u,v\}$ and $H_{\cS}=\{a, b, c, d, e, f, g, h\}$) are shown in Fig.~\ref{fig:grid-example}.

\begin{figure}[!ht] \begin{center}
\vspace{2mm}

{\epsfxsize=5.5in \epsffile{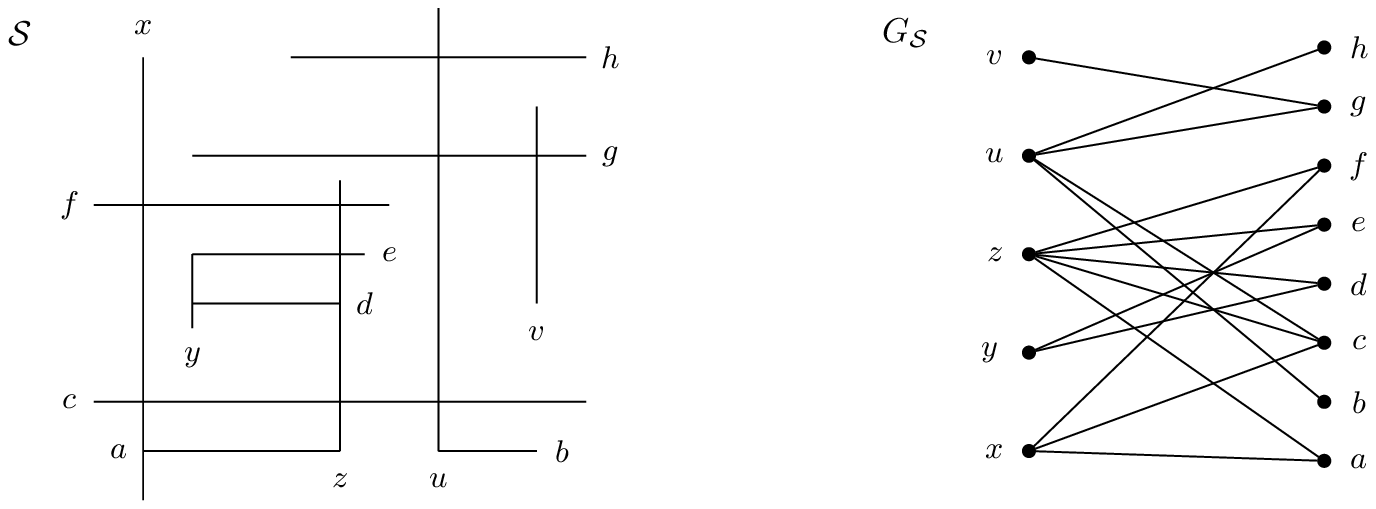}}

\vspace{-3mm}
\caption{Grid $\cS$ and its intersection graph $G_\cS$.\label{fig:grid-example}}
\end{center} \end{figure}

A {\em mobile guard\/} is a guard traveling along a line segment of a grid, and patrolling this line segment and all the intersected line segments.  We identify a mobile guard traveling along a line segment $x$ with the same segment $x$ and say that $x$ patrols itself and the line segments intersected by~$x$. A set $\cC\subseteq \cS$ of mobile guards  is called a {\em patrolling set\/} of the grid $\cS$ if every line segment $x\in {\cal S}$ is either an element of ${\cC}$ or is intersected by an element of ${\cC}$. The problem of mobile guards patrolling a grid is a variant of the traditional art gallery problem and it was formulated by Ntafos~\cite{Nta86}, and then its next variants were studied in a number of papers, see, for example,  \cite{BCCHKLT13}--\cite{FHMPP16}, \cite{KMN03,KPR17} and~\cite{TJ18}. Katz et al.~\cite{KMN03} observed that the decision problem associated with finding the minimum cardinality of a~set of mobile guards patrolling a given grid is \textnormal{NP}-complete. On the other hand it is obvious that all vertical line segments of a grid $\cS$ (as well as all horizontal line segments of $\cS$) form a~patrolling set of $\cS$ and therefore the smallest number of mobile guards patrolling $\cS$ is at most $\min \{|V_\cS|,|H_\cS|\}$. A grid $\cS$ is said to be {\em extremal\/} if the number of mobile guards required to patrol $\cS$ is equal to $\min \{|V_\cS|,|H_\cS|\}$, and we refer to the problem of recognizing extremal grids as the {\em extremal guard cover problem}. From the obvious fact that a~subset $\cC$ of a~grid $\cS$ is a~set of mobile guards patrolling ${\cS}$ if and only if $\cC$ is a dominating set of the intersection graph~$G_\cS$ (see~\cite{KMN03}), it follows that $\cS$ is an extremal grid if and only if $G_\cS$ is a bipartite graph in which the domination number is equal to the cardinality of the smaller of bipartite sets of the graph $G_\cS$. Thus the extremal guard cover problem for a~grid $\cS$ with $n$ line segments can be solved by considering the intersection graph~$G_\cS$ (which can be constructed in $O(n \log n + m)$ time, where $m=O(n^2)$ is the number of intersection points of line segments of $\cS$, see~\cite{Bal95}) and then recognizing whether the domination number of $G_\cS$ is equal to $\min \{|V_\cS|,|H_\cS|\}$, which can be done in $O(n^2)$ time (by Theorem~\ref{thm:algo}). However, since here geometry is involved and intersection graphs of grids form a restricted class of bipartite graphs, some significant improvement in the last statement is possible. We begin with the following lemma.

\begin{lem}\label{lem:upper4}
Let $G_\cS = ((A,B),E_{G_\cS})$ be the intersection graph of a~grid~$\cS$. If $G_\cS$ is leafless and for any two distinct vertices $x$ and $y$ belonging to $A$, and having a common neighbor, there exists a vertex $z$ in~$B$ such that $N_{G_\cS}(z)=  \{x,y\}$, then $\max\{\deg_{G_\cS}(b)\colon b \in B\}  \le 4$.
\end{lem}

\begin{proof} Without loss of generality assume that $A=V_{\cS}$, $B=H_{\cS}$, and suppose that some horizontal line segment $h$ intersects five vertical line segments $v_1, v_2,\ldots, v_5$, say at points $C_1(x_1,y_0), C_2(x_2,y_0),\ldots, C_5(x_5,y_0)$, respectively, with $x_1< x_2< \ldots< x_5$. Let~$\varepsilon$ be a real number such that $0<\varepsilon\le  \min\{x_{i+1}-x_i \colon i=1, \ldots,4 \}/100$. Since the line segment~$h$ (as a vertex of $G_\cS$) is a common neighbor of every two distinct line segments belonging to the set $\{v_1, \ldots, v_5\}$, by assumption for every two indexes $i$ and $j$ ($1\le i<j\le 5$) there exists a horizontal line segment $h_{ij}$ which intersects the line segments $v_i$ and $v_j$ only. Assume that $h_{ij}$ intersects  $v_i$ and $v_j$ at points $L_{ij}(x_i,y_{ij})$ and $R_{ij}(x_j,y_{ij})$, respectively, for some~$y_{ij}$.  These points  define two new points $L^+_{ij}(x_i+\varepsilon,y_{ij})$ and $R^-_{ij}(x_j-\varepsilon,y_{ij})$, respectively,  which lie between $L_{ij}$ and $R_{ij}$ (see Fig.~\ref{fig:K5} for an illustration). Now it follows from the choice of $\varepsilon$ that the points $C_1, C_2,\ldots, C_5$ together with the polygonal chains $C_iL^+_{ij}R^-_{ij}C_j$, where $1\le i<j\le 5$,  form a~planar representation of the complete non-planar graph $K_5$, a~contradiction.
\end{proof}

\begin{figure}[h!]
\begin{center}

{\epsfxsize=2.3in \epsffile{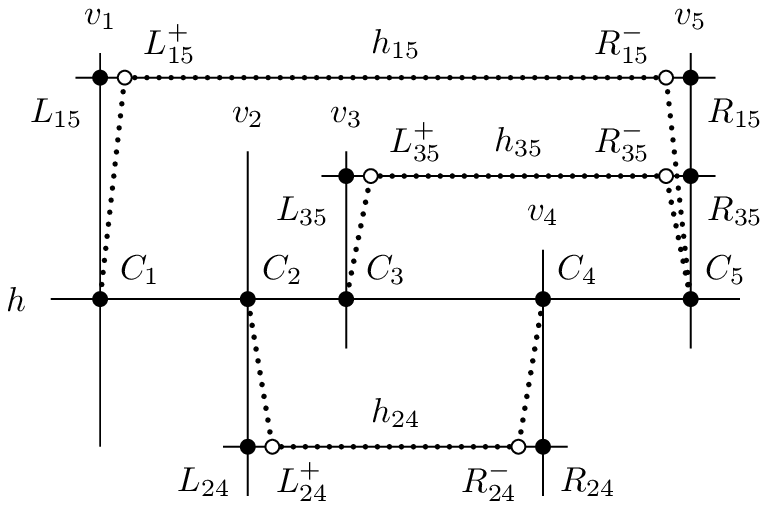}}

\vspace{-1mm}
\caption{}\label{fig:K5}
\end{center}
\end{figure}

The next corollary is immediate from Lemma \ref{lem:upper4} and  Theorem~\ref{thm:main2}.

\begin{cor}\label{cor:4-grids}
Let $G_\cS = ((A,B),E_{G_\cS})$ be the intersection graph of a grid~$\cS$. If $1 \le |A| \le |B|$ and $\gamma(G_{\cS})=|A|$,  then $|N_{G_{\cS}}(b) -(L_{G_{\cS}}\cup S_{G_{\cS}})| \le 4$ for every $b$ in $B$.
\end{cor}

In our last theorem we present the announced time complexity upper bound for the extremal guard cover problem.

\begin{thm}\label{thm:algo_guards}
If ${\cS}$ is a grid with $n$ distinct line segments, then the extremal guard cover problem for ${\cS}$ can be solved in $O(n \log n + m)$ time, where $m$ is the number of intersection points of line segments of $\cS$.\end{thm}

\begin{proof}
Let $G_\cS=((A,B),E_{G_\cS})$ be the intersection graph of $\cS$ with $1 \le |A| \le |B|$. As already mentioned, the graph $G_\cS$ can be constructed in $O(n \log n + m)$ time~\cite{Bal95}, and thus it remains to show that in the same time we can verify whether $\gamma(G_\cS)=|A|$.

First, similarly as in the proof of Theorem~\ref{thm:algo},  if $|A|=|B|$, then we verify whether $G_\cS$ is $C_4$ or a corona graph, which can be done in $O(n+m)$ time. Thus assume $|A| < |B|$. In this case it suffices to verify the properties  (3a) and (3b) in Theorem~\ref{thm:main2}. The first property can be verified in $O(n+m)$ time by identifying and distinctly marking non-leaf and support vertices. To verify the second property, we first determine and distinctly mark the vertices in $A'=A - (L_{G_\cS}\cup S_{G_\cS})$ by using the prior markings. Next, in linear time we check if the inequality $\max_{b \in B} |N(b) - (L_{G_\cS} \cup S_{G_\cS})|  \le 4$ holds for every $b$ in $B$. If $\max_{b \in B} |N(b) - (L_{G_\cS} \cup S_{G_\cS})|  \ge 5$ for some $b\in B$, then we stop as $\gamma(G_\cS) \neq |A|$ (by Corollary~\ref{cor:4-grids}). Otherwise for each vertex in $B$ of degree two whose both neighbors $x$ and $y$ are in $A'$, we list the set $\{x,y\}$. Then we sort the resulting list $L$ of $2$-element sets lexicographically in order to compute the number of sets on the list that occur exactly once. If there is at least one such set, then we stop, as $\gamma(G_\cS) \neq |A|$ (by the property (5b) in Theorem~\ref{thm:main2}).  Otherwise, we continue by truncating the list to store now each set only once.  Since the original list is of length at most~$n$, all that can be done in $O(n \log n)$ time. Next, similarly as in~\cite{AJBT13}, for each vertex $b \in B$, we construct the list $L(b)$ of vertices adjacent to $b$ in $A'$. Since each of these lists is of length at most~$4$ (by the stop condition positively verified above), all that can be done in $O(n)$ time. Finally, using a binary search on $L$, for each $b \in B$ and each $2$-element subset $\{x, y\}$ of $L(b)$ we check whether $\{x,y\}$ is on the (truncated) list $L$, and we stop whenever checking fails (as then $\gamma(G_\cS) \neq |A|$).  This may require checking as many as $\Theta(\sum_{b \in B} |L(b)|^2)$ sets, however, since we handle the case $|L(b)|  \le 4$, the total running time of our algorithm becomes then $O(n \log n+m)$ as required.
\end{proof}

\paragraph{Acknowledgement.} This work was partially supported by National Science Centre, Poland, the grant number 2015/17/B/ST6/01887.

\end{document}